\documentclass[11pt, leqno]{amsart}

\usepackage{amsfonts, amsmath, amssymb, bm, cite, color, mathrsfs, stmaryrd, multirow}
\usepackage[abbrev]{amsrefs}
\usepackage[top=30mm, bottom=30mm, left=25mm, right=25mm]{geometry}

\definecolor{refkey}{rgb}{1,0,0}
\definecolor{labelkey}{rgb}{1,0,0}

\usepackage{mathtools}
\mathtoolsset{showonlyrefs=true}

\numberwithin{equation}{section}

\allowdisplaybreaks[1]

\usepackage{amsthm}
\newtheorem{theo}{Theorem}[section]

\newtheorem{prop}[theo]{Proposition}

\theoremstyle{definition}

\begin{document}
\title[]{An optimal time-singularity of the estimate for the heat semigroup related to the critical Sobolev embedding}

\author[]{Yi C.~Huang}
\address[Yi C.~Huang]{School of Mathematical Sciences, Nanjing Normal University, Nanjing 210023, People's Republic of China}
\email{Yi.Huang.Analysis@gmail.com}

\author[]{Tohru Ozawa}
\address[Tohru Ozawa]{Department of Applied Physics, Waseda University, 3-4-1 Okubo, Shinjuku-ku, Tokyo, 169-8555, Japan}
\email{txozawa@waseda.jp}

\author[]{Chenmin Sun}
\address[Chenmin Sun]{CNRS, Universit\'e Paris-Est Cr\'eteil, Laboratoire d'Analyse et de Math\'ematiques Appliqu\'ees, UMR CNRS 8050, Cr\'eteil, France}
\email{chenmin.sun@cnrs.fr}

\author[]{Taiki Takeuchi}
\address[Taiki Takeuchi]{Institute of Mathematics for Industry, Kyushu University, 744 Motooka, Nishi-ku, Fukuoka, 819-0395, Japan}		
\email{takeuchi.taiki.643@m.kyushu-u.ac.jp}

\thanks{}



\date{}

\maketitle

\begin{abstract}
We give a certain $L^{\infty}(\mathbb{R}^2)$-estimate for the heat semigroup $\{e^{t\Delta}\}_{t \ge 0}$ that is closely related to the fact $H^1(\mathbb{R}^2) \not\subset L^{\infty}(\mathbb{R}^2)$, i.e., the critical Sobolev (non-)embedding and the standard Brezis-Gallou\"et inequality. While we provide several approaches to show such an assertion, we also reveal that the time-singularity of our estimate as $t \to 0^+$ is indeed optimal.
\end{abstract}

\section{Introduction}

This paper aims to verify a certain $L^{\infty}(\mathbb{R}^2)$-estimate for the heat semigroup. We first summarize the notations and definitions in this paper: For a complex-valued function $f$ belonging to the Schwartz space $\mathscr{S}(\mathbb{R}^2)$, we define the Fourier transform $\mathcal{F}$ and its inverse $\mathcal{F}^{-1}$ by setting
$$(\mathcal{F}f)(\xi) \coloneqq \frac{1}{2\pi}\int_{\mathbb{R}^2}e^{-ix\cdot\xi}f(x)dx, \quad (\mathcal{F}^{-1}f)(x) \coloneqq \frac{1}{2\pi}\int_{\mathbb{R}^2}e^{ix\cdot\xi}f(\xi)d\xi.$$
As $\mathcal{F}$ and $\mathcal{F}^{-1}$ are canonically extended to the operators mapping $L^2(\mathbb{R}^2)$ onto itself, the (inhomogeneous) Sobolev spaces $H^s(\mathbb{R}^2)$ are defined by
$$H^s(\mathbb{R}^2) \coloneqq \left\{f \in L^2(\mathbb{R}^2) \, \left| \, \|f\|_{H^s(\mathbb{R}^2)} \coloneqq \left(\int_{\mathbb{R}^2}(1+|\xi|^2)^s|(\mathcal{F}f)(\xi)|^2d\xi\right)^{1/2}<\infty \right.\right\}$$
for $s \ge 0$. We remark that our definition of $\mathcal{F}$ and $\mathcal{F}^{-1}$ implies that
\begin{align}
\label{PI}
\|\mathcal{F}f\|_{L^2(\mathbb{R}^2)} &=\|\mathcal{F}^{-1}f\|_{L^2(\mathbb{R}^2)}=\|f\|_{L^2(\mathbb{R}^2)}, \\
\label{conv}
\mathcal{F}^{-1}[fg] &=\frac{1}{2\pi}(\mathcal{F}^{-1}f)*(\mathcal{F}^{-1}g)
\end{align}
for all $f,g \in L^2(\mathbb{R}^2)$; although such properties are well known, it should be necessary to pay attention to the constants. We also define the heat semigroup $\{e^{t\Delta}\}_{t \ge 0}$ and the function $M:(0,\infty) \to [0,\infty)$ by letting $e^{t\Delta} \coloneqq \mathcal{F}^{-1}e^{-t|\cdot|^2}\mathcal{F}$ and
$$M(t) \coloneqq \sup_{f \in H^1(\mathbb{R}^2) \setminus \{0\}}\frac{\|e^{t\Delta}f\|_{L^{\infty}(\mathbb{R}^2)}}{\|f\|_{H^1(\mathbb{R}^2)}}.$$
Under these definitions, we shall prove the following assertions:

\begin{theo}\label{CSE}
There exists a constant $C>0$ such that
$$\sup_{t>0}\frac{M(t)}{(\log(e+1/t))^{1/2}} \le C.$$
\end{theo}

\begin{theo}\label{optCSE}
There holds
$$\inf_{0<t<1/e}\frac{M(t)}{(\log(e+1/t))^{1/2}} \ge \frac{1}{2\sqrt{2\pi}e^2}.$$
\end{theo}

Let us mention the motivation for these theorems. The second and fourth authors \cite{OT24} recently investigated the standard Gagliardo-Nirenberg and Sobolev inequalities:
$$\|f\|_{L^p(\mathbb{R}^n)} \le C\|f\|_{L^q(\mathbb{R}^n)}^{1-\sigma}\|\nabla f\|_{L^r(\mathbb{R}^n)}^{\sigma}, \quad \|f\|_{L^{2n/(n-2)}(\mathbb{R}^n)} \le C\|\nabla f\|_{L^2(\mathbb{R}^n)},$$
where the first inequality was discussed in the case of $n \ge 1$, $1 \le q \le p \le \infty$, $np/(n+p)<r \le p$, and $\sigma \coloneqq n(1/q-1/p) \{1+n(1/q-1/r)\}^{-1}$ and the second one was in the case of $n \ge 3$. The aim of the preceding paper was to give an alternative proof of these inequalities by using the heat semigroup $\{e^{t\Delta}\}_{t \ge 0}$. The Gagliardo-Nirenberg inequality has been shown by combining the following simple identity
\begin{equation}\label{heatdecomp}
f=e^{t\Delta}f-\int_0^t\Delta e^{\tau\Delta}fd\tau
\end{equation}
with the standard $L^q$-$L^p$ estimate for the heat semigroup
\begin{equation}\label{heatest}
\|\nabla^je^{t\Delta}f\|_{L^p(\mathbb{R}^n)} \le Ct^{-(n/2)(1/q-1/p)-j/2}\|f\|_{L^q(\mathbb{R}^n)}
\end{equation}
for $t>0$, $j \in \{0,1\}$, and $1 \le q \le p \le \infty$ (cf.~\cite{GGS}*{Theorem, p.8}). However, for the case of the Sobolev inequality, the same method may \textit{no longer} be applicable, and thus another approach via the Hardy inequality has been taken. Whereas the preceding paper \cite{OT24} has been concerned with the Gagliardo-Nirenberg and Sobolev inequalities, the present paper aims to investigate the standard \textit{Brezis-Gallou\"et inequality}:
\begin{equation}\label{BG}
\|f\|_{L^{\infty}(\mathbb{R}^2)} \le C\|f\|_{H^1(\mathbb{R}^2)}\left\{1+\left(\log\left(1+\frac{\|f\|_{H^2(\mathbb{R}^2)}}{\|f\|_{H^1(\mathbb{R}^2)}}\right)\right)^{1/2}\right\}.
\end{equation}
Namely, our eventual goal is to provide its proof based on the heat semigroup $\{e^{t\Delta}\}_{t \ge 0}$ approach. Similarly to the case of the Sobolev inequality, we have difficulty in the proof from the simple combination of \eqref{heatdecomp} and \eqref{heatest}; one might expect that the appearance of the logarithmic function in \eqref{BG} is due to the $\tau^{-1}$-singularity appearing in \eqref{heatdecomp}. In other words, we calculate like
$$\left\|\int_1^t\Delta e^{\tau\Delta}fd\tau\right\|_{L^{\infty}(\mathbb{R}^2)} \le C\int_1^t\tau^{-1}\|\nabla f\|_{L^2(\mathbb{R}^2)}d\tau \le C(\log t)\|f\|_{H^1(\mathbb{R}^2)}$$
for $t>1$ by relying on \eqref{heatest} with $j=1$, $p=\infty$, and $q=2$. However, such an estimate is valid only for $t>1$; the same estimate does \textit{not} work near $t=0$. Next, we consider another way, namely, we simply combine \eqref{heatdecomp} and \eqref{heatest} again to observe that
\begin{equation}
\begin{aligned}
\|f\|_{L^{\infty}(\mathbb{R}^2)} &\le \|e^{t\Delta}f\|_{L^{\infty}(\mathbb{R}^2)}+\int_0^t\|e^{\tau\Delta}\Delta f\|_{L^{\infty}(\mathbb{R}^2)}d\tau \\
&\le \|e^{t\Delta}f\|_{L^{\infty}(\mathbb{R}^2)}+\int_0^tC\tau^{-1/2}\|\Delta f\|_{L^2(\mathbb{R}^2)}d\tau \\
&\le \|e^{t\Delta}f\|_{L^{\infty}(\mathbb{R}^2)}+Ct^{1/2}\|f\|_{H^2(\mathbb{R}^2)}
\end{aligned}
\end{equation}
for all $t>0$. Even in this case, we find that the estimate \eqref{heatest} is \textit{insufficient} to obtain \eqref{BG}. However, if we assume that the estimate in Theorem \ref{CSE} holds true, then we have
$$\|e^{t\Delta}f\|_{L^{\infty}(\mathbb{R}^2)} \le C(\log(e+1/t))^{1/2}\|f\|_{H^1(\mathbb{R}^2)}.$$
Thus we can show that
$$\|f\|_{L^{\infty}(\mathbb{R}^2)} \le C\|f\|_{H^1(\mathbb{R}^2)}\left\{(\log(e+1/t))^{1/2}+t^{1/2}\frac{\|f\|_{H^2(\mathbb{R}^2)}}{\|f\|_{H^1(\mathbb{R}^2)}}\right\},$$
which leads to \eqref{BG} with a different constant $C>0$ by letting $t=\|f\|_{H^1(\mathbb{R}^2)}^2\|f\|_{H^2(\mathbb{R}^2)}^{-2}$. This means that Theorem \ref{CSE}, which is our first result, provides a proof of the Brezis-Gallou\"et inequality \eqref{BG} through the heat semigroup $\{e^{t\Delta}\}_{t \ge 0}$.

\subsection{Several remarks on our main results}

Here we shall state some of the remarks:

\begin{enumerate}
\item It is known that the Brezis-Gallou\"et inequality \eqref{BG} may be regarded as a \textit{critical} Sobolev embedding; as the Sobolev embedding yields $H^s(\mathbb{R}^n) \subset L^{\infty}(\mathbb{R}^n)$ whenever $s>n/2$, letting $n=2$ implies that $H^s(\mathbb{R}^2) \subset L^{\infty}(\mathbb{R}^2)$ for $s>1$ but $H^1(\mathbb{R}^2) \not\subset L^{\infty}(\mathbb{R}^2)$. Although \eqref{BG} requires the assumption $f \in H^2(\mathbb{R}^2)$, the dependence of the growth for $\|f\|_{H^2(\mathbb{R}^2)} \gg 1$ is at most $O((\log \|f\|_{H^2(\mathbb{R}^2)})^{1/2})$. As well as \eqref{BG}, our estimate in Theorem \ref{CSE} states that we may no longer expect that $e^{t\Delta}f \in L^{\infty}(\mathbb{R}^2)$ is still valid as $t \to 0^+$ due to the critical Sobolev (non-)embedding. However, such a time-singularity is given at most $O((\log(1/t))^{1/2})$.

\item A certain (numerical) constant $C>0$ appears in the estimate of Theorem \ref{CSE}. In fact, we will observe that $C>0$ is bounded from above by $(2\pi)^{-1/2}$, which does not seem to be the best constant; the details will be discussed in Subsection \ref{proof3} below.

\item Although we do not find the best constant, we may observe that the time-singularity of the estimate in Theorem \ref{CSE} as $t \to 0^+$ is indeed \textit{optimal}. Such an assertion is given as Theorem \ref{optCSE}. We should remark that, in the same vein, Brezis and Wainger \cite{BW80}*{Remark 2} have verified the optimality of the power $1-1/p$ appearing in \eqref{BW} below.
\end{enumerate}

\subsection{Previous works related to our results}

We shall summarize the previous works on the Brezis-Gallou\"et inequality \eqref{BG}: The original inequality was established by Brezis and Gallou\"et \cite{BG80} in 1980 in constructing global solutions of the 2D nonlinear Schr\"odinger equations. Because of the \textit{slow} growth of the $H^2$-norm in \eqref{BG}, they achieved controlling the nonlinear term $|u|^2u$ and obtaining the uniform a priori estimate in the $H^2$-framework. Brezis and Wainger \cite{BW80} generalized the original inequality \eqref{BG} through the fractional Sobolev spaces $H^{s,q}(\mathbb{R}^n)$; they showed that
\begin{equation}\label{BW}
\|f\|_{L^{\infty}(\mathbb{R}^n)} \le C\|f\|_{H^{n/p,p}(\mathbb{R}^n)}\left\{1+\left(\log\left(1+\frac{\|f\|_{H^{s,q}(\mathbb{R}^n)}}{\|f\|_{H^{n/p,p}(\mathbb{R}^n)}}\right)\right)^{1-1/p}\right\}
\end{equation}
for all $1<p<\infty$, $1 \le q \le \infty$, and $s>n/q$. The assumption $s>n/q$ is crucial as it guarantees $H^{s,q}(\mathbb{R}^n) \subset L^{\infty}(\mathbb{R}^n)$. Beale, Kato, and Majda \cite{BKM} treated the 3D Euler equations differing from the 2D nonlinear Schr\"odinger equations; via showing the inequality
\begin{equation}\label{rotest}
\|\nabla f\|_{L^{\infty}(\mathbb{R}^3)} \le C\Big\{1+\|\mathrm{rot} \, f\|_{L^2(\mathbb{R}^3)}+\|\mathrm{rot} \, f\|_{L^{\infty}(\mathbb{R}^3)}\log(e+\|f\|_{H^3(\mathbb{R}^3)})\Big\},
\end{equation}
which is similar to \eqref{BW}, they obtained the criterion of whether the local-in-time solution may be extended. Engler \cite{HE89} gave another proof of the inequality \eqref{BW}. In particular, the method of \cite{HE89} provides a space-localized version of \eqref{BW}, which simply leads to the standard case \eqref{BW}. Based on the idea of \cite{HE89}, the second author \cite{TO95} also gave another proof of \eqref{BW}; compared with \cite{HE89}, the result of \cite{TO95} is still valid even for the fractional Sobolev spaces. Kozono and Taniuchi \cite{KT00} refined \eqref{rotest} to establish
\begin{equation}\label{rotest2}
\|f\|_{L^{\infty}(\mathbb{R}^n)} \le C\Big\{1+\|f\|_{BMO(\mathbb{R}^n)}\log(e+\|f\|_{H^{s,p}(\mathbb{R}^n)})\Big\}
\end{equation}
for $1<p<\infty$ and $s>n/p$, where $BMO(\mathbb{R}^n)$ denotes the space of functions of bounded mean oscillation. Applying it to the Euler equations, they also improved the result of \cite{BKM}. The further generalizations through the Besov-type spaces have been given by, e.g., Kozono, Ogawa, and Taniuchi \cite{KOTB}, Chae \cite{DC02}, Ogawa \cite{TO03}, and Ogawa and Taniuchi \cite{OT04}. For the case of the more involved function spaces, we refer to the result of Nakao and Taniuchi \cite{NT18}. We also remark that there is another development for the classical result of Brezis and Gallou\"et \cite{BG80}; the second author and Visciglia \cite{OV16} extended the global existence result for the 2D nonlinear Schr\"odinger equations in \cite{BG80} by replacing the nonlinear term $|u|^2u$ with $|u|^3u$. The method is based \textit{not} on the refined version of \eqref{BG} but on the introduction of the \textit{new} energy structure and quite precise calculations.

Next, we mention the works on the semigroup approach for the Sobolev-type inequalities: As stated above, while the second and fourth authors \cite{OT24} investigated the Gagliardo-Nirenberg and Sobolev inequalities, this paper focuses on the Brezis-Gallou\"et inequality \eqref{BG}. We have aimed to take the heat semigroup approach to show these inequalities. It should be noted that the semigroup approach has been taken previously. For instance, Carlen, Kusuoka, and Stroock \cite{CKS} showed that the $L^q$-$L^p$ estimate for the heat semigroup \eqref{heatest} and the Nash inequality \cite{JN58}:
$$\|f\|_{L^2(\mathbb{R}^n)} \le C\|f\|_{L^1(\mathbb{R}^n)}^{2/(n+2)}\|\nabla f\|_{L^2(\mathbb{R}^n)}^{1-2/(n+2)},$$
which is a special case of the Gagliardo-Nirenberg inequality, are equivalent. Varopoulos, Saloff-Coste, and Coulhon \cite{VSC}*{II.3.3 Remarks (a)} obtained the Nash inequality via the semigroup estimate as well (see also \cite{LS02}*{Section 4.1}). Concerning the case of the Gagliardo-Nirenberg or Sobolev inequality, we refer to, e.g., Varopoulos \cite{NV85}, Maremonti \cite{PM98}*{Theorem 2.2}, Giga, Giga, and Saal \cite{GGS}*{Theorem, p.190}, and Bakry, Gentil, and Ledoux \cite{BGL}*{Section 6.3}.

Finally, we also refer to some works on Sobolev-type inequalities in a general domain $\Omega \subset \mathbb{R}^n$ or metric space: Although we fail to include every corner of the previous efforts because of the numerous results, we mention that Maz'ya \cite{VM85}*{Corollary, p.169 and Theorem 2, p.211} obtained the necessary and sufficient conditions on a domain $\Omega$ for the validity of the Sobolev-type embeddings (see also \cite{VM60}). Brezis and Lieb \cite{BL85} considered a bounded domain $\Omega$ and gave the Sobolev inequality with an extra (boundary integral) term
$$S_n^{1/2}\|f\|_{L^{2n/(n-2)}(\Omega)} \le \|\nabla f\|_{L^2(\Omega)}+C\|f\|_{L^p(\partial\Omega)}$$
for $n \ge 3$ with $p \coloneqq 2(n-1)/(n-2)$, where $S_n>0$ denotes the best Sobolev embedding constant given by Aubin \cite{TA76} and Talenti \cite{GT76}. They \cite{BL85} also gave several variants of the above inequality. For the Sobolev-type inequalities on irregular domains, we refer to the works of, e.g., Haj{\l}asz and Koskela \cite{HK98}, Kilpel\"ainen and Mal\'y \cite{KM00}, and Cianchi and Maz'ya \cite{CM16}. Concerning the case of a Riemannian manifold $M$, Rothaus \cite{OR85}*{Theorem 1} gave the necessary and sufficient conditions on $M$ for the validity of a certain estimate related to the Sobolev-type embedding (see also \cite{LS02}*{Theorem 3.1.2}). Saloff-Coste \cite{LS92}*{Theorem 10.3} applied the result of Varopoulos \cite{NV85} to show that
$$\|f\|_{L^{np/(n-p)}(B(x,r))} \le e^{C(1+\sqrt{K(x,r)}r)}V(x,r)^{-1/n}\Big(r\|\nabla f\|_{L^p(B(x,r))}+\|f\|_{L^p(B(x,r))} \Big)$$
for all $1 \le p <n$, where $B(x,r) \subset M$, $V(x,r)>0$, and $K(x,r)>0$ denote a ball centered at $x \in M$ with a radius $r>0$, its volume, and the lower bound of the Ricci curvature in $B(x,2r)$, respectively. For further details, we refer to the books of Varopoulos, Saloff-Coste, and Coulhon \cite{VSC}*{Section IV.7 and Chapter IX}, Hebey \cite{EH99}*{Chapters 2 and 3}, and Saloff-Coste \cite{LS02}*{Chapter 3}. In addition, Haj{\l}asz \cite{PH96} extended the Sobolev spaces by means of the metric space and investigated their embedding properties. Bobkov and Houdr\'e \cite{BH97}*{Theorem 1.1} obtained the metric space version of \cite{OR85}*{Theorem 1}. Last but not least, we also mention the result of Bahouri, Fermanian-Kammerer, and Gallagher \cite{BFG} who gave the Sobolev-type inequalities in the Besov spaces on the graded Lie groups.

\section{Several approaches to the proof of Theorem \ref{CSE}}

\subsection{Proof by relying on the results in the Besov spaces}

We first remark that Theorem \ref{CSE} is actually a simple consequence of the estimate in the Besov spaces framework. Recall the definition of the Besov spaces briefly here; let $\varphi \in \mathscr{S}(\mathbb{R}^2)$ satisfy $\mathrm{supp} \, \varphi=\{\xi \in \mathbb{R}^2 \, | \, 1/2 \le |\xi| \le 2\}$, $\varphi(\xi)>0$ for $1/2<|\xi|<2$, and $\sum_{j \in \mathbb{Z}}\varphi(2^{-j}\xi)=1$ for all $\xi \in \mathbb{R}^2 \setminus \{0\}$. We moreover set $\psi(\xi) \coloneqq 1-\sum_{j \in \mathbb{N}}\varphi(2^{-j}\xi)$ for $\xi \in \mathbb{R}^2$ and define
$$\|f\|_{B_{p,q}^s(\mathbb{R}^2)} \coloneqq \|(\mathcal{F}^{-1}\psi)*f\|_{L^p(\mathbb{R}^2)}+\|\{2^{sj}\|\mathcal{F}^{-1}[\varphi(2^{-j}\cdot)]*f\|_{L^p(\mathbb{R}^2)}\}_{j \in \mathbb{N}}\|_{l^q(\mathbb{N})}$$
for $1 \le p \le \infty$, $s \in \mathbb{R}$, and $1 \le q \le \infty$. The Besov spaces $B_{p,q}^s(\mathbb{R}^2)$ are defined as the subspaces of the tempered distribution $\mathscr{S}'(\mathbb{R}^2)$ consisting of those functions $f$ such that $\|f\|_{B_{p,q}^s(\mathbb{R}^2)}<\infty$. For more details on the Besov spaces, we refer to the books written by Bergh and L\"ofstr\"om \cite{BL76}*{Section 6.2} and Bahouri, Chemin, and Danchin \cite{BCD}*{Section 2.7}.

\begin{proof}[Proof of Theorem \ref{CSE}]
We note the following embedding properties for the Besov spaces \cite{BL76}*{Theorems 6.2.4 and 6.5.1}:
$$B_{2,1}^1(\mathbb{R}^2) \subset B_{\infty,1}^0(\mathbb{R}^2) \subset L^{\infty}(\mathbb{R}^2).$$
Hence, we recall the result shown by Kozono, Ogawa, and Taniuchi \cite{KOT}*{Lemma 2.2 (i)} to obtain
$$\|e^{t\Delta}f\|_{L^{\infty}(\mathbb{R}^2)} \le C\|e^{t\Delta}f\|_{B_{2,1}^1(\mathbb{R}^2)} \le C\log(e+1/t)\|f\|_{B_{2,\infty}^1(\mathbb{R}^2)}$$
for all $t>0$ and $f \in B_{2,\infty}^1(\mathbb{R}^2)$, while we have
$$\|e^{t\Delta}f\|_{L^{\infty}(\mathbb{R}^2)} \le \|f\|_{L^{\infty}(\mathbb{R}^2)} \le C\|f\|_{B_{2,1}^1(\mathbb{R}^2)}$$
for all $t>0$ and $f \in B_{2,1}^1(\mathbb{R}^2)$. Since the real interpolation space of the Besov spaces is given by
$$(B_{2,\infty}^1(\mathbb{R}^2),B_{2,1}^1(\mathbb{R}^2))_{1/2,2}=B_{2,2}^1(\mathbb{R}^2)=H^1(\mathbb{R}^2)$$
due to \cite{BL76}*{Theorems 6.4.4 and 6.4.5}, the standard property of the real interpolation \cite{AL09}*{Theorem 1.6} provides the desired estimate.
\end{proof}

As one may see, although we may show Theorem \ref{CSE} by combining the known properties, the above method is based on the theory of Besov spaces and the real interpolation technique. In the introduction, we have stated that our actual goal is to show the Brezis-Gallou\"et inequality \eqref{BG} by a simple approach. However, the above proof relies on the more \textit{advanced} methods in comparison with a direct proof of \eqref{BG}; for such a reason, hereafter we give a direct proof of Theorem \ref{CSE} \textit{without} relying on any advanced results or involved notions.

\subsection{Proof based on the dyadic decomposition approach}

One of the methods for a direct proof of Theorem \ref{CSE} is the dyadic decomposition. Similarly to the proof of \cite{KOT}*{Lemma 2.1 (i)}, we split the summation suitably to derive the logarithmic singularity. Our proof also shares the same spirit as the known proof of the Brezis-Gallou\"et inequality \cite{KOTB}*{Theorem 2.1 (1)}. However, while the original argument relies on the Bernstein inequality \cite{BCD}*{Lemma 2.1}, our approach leverages heat semigroup estimates as the key ingredient. In addition, as we will use only the Sobolev space $H^1(\mathbb{R}^2)$ instead of the Besov-type space, we may reveal an \textit{explicit} bound of the constant $C>0$ appearing in the estimate of Theorem \ref{CSE}.

\begin{proof}[Proof of Theorem \ref{CSE}]
Define the sequence $\{\chi_j\}_{j \ge 0} \subset L^{\infty}(0,\infty)$ of functions by setting
$$\chi_0 \coloneqq \left\{\begin{array}{cl}
1 & \text{on $(0,1)$}, \\
0 & \text{on $[1,\infty)$},
\end{array}\right. \quad \chi_j \coloneqq \left\{\begin{array}{cl}
1 & \text{on $(2^{j-1},2^j)$}, \\
0 & \text{on $(0,2^{j-1}] \cup [2^j,\infty)$}
\end{array}\right.$$
for $j \in \mathbb{N}$. Since
\begin{equation}\label{decomp}
\begin{aligned}
e^{t\Delta}f &=\mathcal{F}^{-1}\left[\sum_{j \ge 0}\chi_j(|\cdot|)e^{-t|\cdot|^2}\mathcal{F}f\right] \\
&=\sum_{j=0}^N\mathcal{F}^{-1}[\chi_j(|\cdot|)^2e^{-t|\cdot|^2}\mathcal{F}f]+\mathcal{F}^{-1}\left[\sum_{j=N+1}^{\infty}\chi_j(|\cdot|)e^{-t|\cdot|^2}\mathcal{F}f\right]
\end{aligned}
\end{equation}
for all $t>0$ and $N \in \mathbb{N}$ due to $\chi_j=\chi_j^2$ and since the Cauchy-Schwarz inequality implies that
\begin{equation}
\begin{aligned}
&\left\|\mathcal{F}^{-1}\left[\sum_{j=N+1}^{\infty}\chi_j(|\cdot|)e^{-t|\cdot|^2}\mathcal{F}f\right]\right\|_{L^{\infty}(\mathbb{R}^2)} \\
&\le \frac{1}{2\pi}\left\|\sum_{j=N+1}^{\infty}\chi_j(|\cdot|)e^{-t|\cdot|^2}\mathcal{F}f\right\|_{L^1(\mathbb{R}^2)} \\
&\le \frac{1}{2\pi}\exp(-2^{2N-1}t)\|e^{-(t/2)|\cdot|^2}\|_{L^2(\mathbb{R}^2)}\|\mathcal{F}f\|_{L^2(\mathbb{R}^2)}
\end{aligned}
\end{equation}
for all $t>0$ and $N \in \mathbb{N}$, by taking the $L^{\infty}(\mathbb{R}^2)$-norm and letting $N \to \infty$ in the estimate \eqref{decomp}, we observe that
\begin{equation}
\begin{aligned}
\|e^{t\Delta}f\|_{L^{\infty}(\mathbb{R}^2)} &\le \sum_{j \ge 0}\|\mathcal{F}^{-1}[\chi_j(|\cdot|)^2e^{-t|\cdot|^2}\mathcal{F}f]\|_{L^{\infty}(\mathbb{R}^2)} \\
&\le \frac{1}{2\pi}\sum_{j \ge 0}\|\chi_j(|\cdot|)^2e^{-t|\cdot|^2}\mathcal{F}f\|_{L^1(\mathbb{R}^2)} \\
&\le \frac{1}{2\pi}\sum_{j \ge 0}\|\chi_j(|\cdot|)\|_{L^2(\mathbb{R}^2)}\|\chi_j(|\cdot|)e^{-t|\cdot|^2}\mathcal{F}f\|_{L^2(\mathbb{R}^2)}
\end{aligned}
\end{equation}
with the aid of the Cauchy-Schwarz inequality. In addition, noting that
\begin{equation}
\begin{alignedat}{4}
\|\chi_0(|\cdot|)\|_{L^2(\mathbb{R}^2)}^2 &=\int_{|\xi|<1}d\xi &&=2\pi \int_0^1rdr &&= \pi, \\
\|\chi_j(|\cdot|)\|_{L^2(\mathbb{R}^2)}^2 &=\int_{2^{j-1} \le |\xi|<2^j}d\xi &&=2\pi \int_{2^{j-1}}^{2^j}rdr &&\le \pi \cdot 2^{2j}
\end{alignedat}
\end{equation}
for all $j \in \mathbb{N}$, we obtain
\begin{equation}\label{preest}
\|e^{t\Delta}f\|_{L^{\infty}(\mathbb{R}^2)} \le \frac{1}{2\pi}\sum_{j \ge 0}\sqrt{\pi \cdot 2^{2j}} \cdot \|\chi_j(|\cdot|)e^{-t|\cdot|^2}\mathcal{F}f\|_{L^2(\mathbb{R}^2)}.
\end{equation}

As there holds
\begin{equation}\label{auxest}
\begin{aligned}
\sum_{j \in \mathbb{N}}2^{2\kappa j}\|\chi_j(|\cdot|)|\cdot|^{1-\kappa}\mathcal{F}f\|_{L^2(\mathbb{R}^2)}^2 &=\sum_{j \in \mathbb{N}}2^{2\kappa j}\int_{2^{j-1} \le |\xi|<2^j}|\xi|^{2-2\kappa}|(\mathcal{F}f)(\xi)|^2d\xi \\
&\le 2^{2\kappa}\sum_{j \in \mathbb{N}}\int_{2^{j-1} \le |\xi|<2^j}|\xi|^2|(\mathcal{F}f)(\xi)|^2d\xi \\
&\le 2^{2\kappa}\||\cdot|\mathcal{F}f\|_{L^2(\mathbb{R}^2)}^2
\end{aligned}
\end{equation}
for all $\kappa \ge 0$, we rely on the estimate \eqref{auxest} with $\kappa=1$ to deduce that
\begin{equation}
\begin{aligned}
\sum_{j=1}^N2^j\|\chi_j(|\cdot|)e^{-t|\cdot|^2}\mathcal{F}f\|_{L^2(\mathbb{R}^2)} &\le \sqrt{N}\left(\sum_{j=1}^N2^{2j}\|\chi_j(|\cdot|)\mathcal{F}f\|_{L^2(\mathbb{R}^2)}^2\right)^{1/2} \\
&\le 2\sqrt{N}\||\cdot|\mathcal{F}f\|_{L^2(\mathbb{R}^2)}
\end{aligned}
\end{equation}
for all $t>0$ and $N \in \mathbb{N}$. We also combine the simple fact $\max_{\tau>0}\tau e^{-\tau^2}=(2e)^{-1/2}$ and the estimate \eqref{auxest} with $\kappa=2$ to obtain
\begin{equation}
\begin{aligned}
&\sum_{j=N+1}^{\infty}2^j\|\chi_j(|\cdot|)e^{-t|\cdot|^2}\mathcal{F}f\|_{L^2(\mathbb{R}^2)} \\
&\le \left(\sum_{j=N+1}^{\infty}2^{-2j}\right)^{1/2}\left(\sum_{j=N+1}^{\infty}2^{4j}\|\chi_j(|\cdot|)e^{-t|\cdot|^2}\mathcal{F}f\|_{L^2(\mathbb{R}^2)}^2\right)^{1/2} \\
&\le 2^{-N}\left(\frac{1}{2e}\sum_{j=N+1}^{\infty}2^{4j}\|\chi_j(|\cdot|)(t^{-1/2}|\cdot|^{-1})\mathcal{F}f\|_{L^2(\mathbb{R}^2)}^2\right)^{1/2} \\
&\le 2^{-N+1}t^{-1/2}\||\cdot|\mathcal{F}f\|_{L^2(\mathbb{R}^2)}
\end{aligned}
\end{equation}
for all $t>0$ and $N \in \mathbb{N}$. Hence, we may estimate \eqref{preest} as
\begin{equation}
\begin{aligned}
\|e^{t\Delta}f\|_{L^{\infty}(\mathbb{R}^2)} &\le \frac{1}{2\sqrt{\pi}}\left( \|\chi_0(|\cdot|)e^{-t|\cdot|^2}\mathcal{F}f\|_{L^2(\mathbb{R}^2)}+\sum_{j \in \mathbb{N}}2^j\|\chi_j(|\cdot|)e^{-t|\cdot|^2}\mathcal{F}f\|_{L^2(\mathbb{R}^2)} \right) \\
&\le \frac{1}{2\sqrt{\pi}}\left( \|\mathcal{F}f\|_{L^2(\mathbb{R}^2)}+2\sqrt{N}\||\cdot|\mathcal{F}f\|_{L^2(\mathbb{R}^2)}+2^{-N+1}t^{-1/2}\||\cdot|\mathcal{F}f\|_{L^2(\mathbb{R}^2)}\right) \\
&\le \frac{2}{\sqrt{\pi}}(\sqrt{N}+2^{-N}t^{-1/2})\|f\|_{H^1(\mathbb{R}^2)}
\end{aligned}
\end{equation}
for all $t>0$ and $N \in \mathbb{N}$.

We now suppose that $t \ge 1$. Then, taking $N=1$ implies the desired estimate. Next, we suppose that $0<t<1$. Taking $N \in \mathbb{N}$ so that
$$\frac{\log t^{-1/2}}{\log 2} \le N \le 1+\frac{\log t^{-1/2}}{\log 2},$$
we observe that $t^{-1/2} \le 2^N$ and $N \le 1+\log(1/t) \le 2\log(e+1/t)$, which lead to
\begin{equation}
\begin{aligned}
\|e^{t\Delta}f\|_{L^{\infty}(\mathbb{R}^2)} &\le \frac{2}{\sqrt{\pi}}\left\{\sqrt{2}(\log(e+1/t))^{1/2}+1\right\}\|f\|_{H^1(\mathbb{R}^2)} \\
&\le 4(\log(e+1/t))^{1/2}\|f\|_{H^1(\mathbb{R}^2)}.
\end{aligned}
\end{equation}
This completes the proof of Theorem \ref{CSE}.
\end{proof}

\subsection{Proof based on the simpler direct computation}
\label{proof3}

The authors realized the \textit{simpler} approach after giving an alternative proof of Theorem \ref{CSE} above. In addition, this method might give a sharp estimate. The following result is essential to complete the proof of Theorem \ref{CSE}:

\begin{prop}\label{preCSE}
There holds
$$\|e^{t\Delta}f\|_{L^{\infty}(\mathbb{R}^2)} \le \frac{e^t}{2\sqrt{\pi}}(E_1(2t))^{1/2}\|f\|_{H^1(\mathbb{R}^2)}$$
for all $t>0$ and $f \in H^1(\mathbb{R}^2)$, where $E_1$ denotes the exponential integral defined by
$$E_1(\tau) \coloneqq \int_{\tau}^{\infty}\frac{e^{-\rho}}{\rho}d\rho, \quad \tau>0.$$
\end{prop}

\begin{proof}
As we have
\begin{equation}
\begin{aligned}
e^{t\Delta}f &=\mathcal{F}^{-1}[(1+|\cdot|^2)^{-1/2}e^{-t|\cdot|^2}(1+|\cdot|^2)^{1/2}\mathcal{F}f] \\
&=\frac{1}{2\pi}\mathcal{F}^{-1}[(1+|\cdot|^2)^{-1/2}e^{-t|\cdot|^2}]*\mathcal{F}^{-1}[(1+|\cdot|^2)^{1/2}\mathcal{F}f]
\end{aligned}
\end{equation}
for all $t>0$ due to \eqref{conv}, the Young convolution inequality and \eqref{PI} imply that
\begin{equation}
\begin{aligned}
\|e^{t\Delta}f\|_{L^{\infty}(\mathbb{R}^2)}^2 &\le \frac{1}{(2\pi)^2}\|\mathcal{F}^{-1}[(1+|\cdot|^2)^{-1/2}e^{-t|\cdot|^2}]\|_{L^2(\mathbb{R}^2)}^2\|\mathcal{F}^{-1}[(1+|\cdot|^2)^{1/2}\mathcal{F}f]\|_{L^2(\mathbb{R}^2)}^2 \\
&\le \frac{1}{(2\pi)^2}\int_{\mathbb{R}^2}\frac{e^{-2t|\xi|^2}}{1+|\xi|^2}d\xi \cdot \|f\|_{H^1(\mathbb{R}^2)}^2.
\end{aligned}
\end{equation}
Since the substitution $\rho=2t(1+r^2)$ yields $dr/d\rho=1/(4tr)$, we deduce that
\begin{equation}
\begin{aligned}
\frac{1}{(2\pi)^2}\int_{\mathbb{R}^2}\frac{e^{-2t|\xi|^2}}{1+|\xi|^2}d\xi &=\frac{1}{2\pi}\int_0^{\infty}\frac{e^{-2tr^2}}{1+r^2} \cdot rdr \\
&=\frac{e^{2t}}{2\pi}\int_{2t}^{\infty}e^{-2t(1+r^2)} \cdot \frac{2t}{\rho} \cdot r\frac{dr}{d\rho}d\rho \\
&=\frac{e^{2t}}{4\pi}\int_{2t}^{\infty}\frac{e^{-\rho}}{\rho}d\rho=\frac{e^{2t}}{4\pi}E_1(2t)
\end{aligned}
\end{equation}
for all $t>0$, which leads to the desired estimate.
\end{proof}

The above method relies \textit{only} on the direct calculation without extra estimation steps. Hence, it is expected that the estimate in Proposition \ref{preCSE} might be sharp. However, we have to use the notion of the exponential integral $E_1$ that may \textit{not} be represented by the elementary functions. Once we show Proposition \ref{preCSE}, the remaining step is just revealing the dependence of the constant on $t>0$. It should be noted that although Theorem \ref{CSE} may be simply obtained by Proposition \ref{preCSE}, the sharpness of the bound $C>0$ would be lost due to the rough estimation for $E_1$.

\begin{proof}[Proof of Theorem \ref{CSE}]
While we have
\begin{equation}\label{E11}
E_1(2t) \le \frac{1}{2t}\int_{2t}^{\infty}e^{-\rho}d\rho=\frac{1}{2t}e^{-2t}
\end{equation}
for all $t>0$, we note $\max_{\tau>0}(e\tau+2)e^{-\tau}=e^{2/e}$ to observe that
\begin{equation}\label{E12}
E_1(2t) \le e^{2/e}\int_{2t}^{\infty}\frac{1}{\rho(e\rho+2)}d\rho=e^{2/e}\left[-\frac{1}{2}\log(e+2/\rho)\right]_{2t}^{\infty} \le \frac{e^{2/e}}{2}\log(e+1/t)
\end{equation}
for all $t>0$. In the case of $t \ge 2/3-1/e$, it holds by $1+\tau+\tau^2/2 \le e^{\tau}$ for $\tau>0$ that $t \ge 2/3-2/5>1/4$, and thus \eqref{E11} yields
$$\frac{e^t}{2\sqrt{\pi}}(E_1(2t))^{1/2}<\frac{e^t}{2\sqrt{\pi}} \cdot \sqrt{2}e^{-t}=\frac{1}{\sqrt{2\pi}}.$$
In the case of $0<t<2/3-1/e$, we see by $e^{\tau} \le (2+\tau)/(2-\tau)$ for $0<\tau<2$ that $e^{2/3}<2$, and thus \eqref{E12} implies that
$$\frac{e^t}{2\sqrt{\pi}}(E_1(2t))^{1/2} \le \frac{e^{2/3-1/e}}{2\sqrt{\pi}} \cdot \frac{e^{1/e}}{\sqrt{2}}(\log(e+1/t))^{1/2}<\frac{1}{\sqrt{2\pi}}(\log(e+1/t))^{1/2}.$$
Therefore, the desired estimate follows with the bound $C \le (2\pi)^{-1/2}$.
\end{proof}

We have shown Proposition \ref{preCSE} to derive Theorem \ref{CSE} and the simplest version of the Brezis-Gallou\"et inequality \eqref{BG}. From its method, we easily observe that the results of Proposition \ref{preCSE} and Theorem \ref{CSE} can be \textit{extended} as follows:

\begin{theo}\label{CSEq}
Let $n \ge 1$, $1 \le q \le 2$, and $s \in \mathbb{R}$. Suppose that $f \in \mathscr{S}'(\mathbb{R}^n)$ satisfies
$$\|f\|_{H^{s,q}(\mathbb{R}^n)} \coloneqq \|\mathcal{F}^{-1}[(1+|\cdot|^2)^{s/2}\mathcal{F}f]\|_{L^q(\mathbb{R}^n)}<\infty.$$
Then it holds that
\begin{equation}
\|e^{t\Delta}f\|_{L^{\infty}(\mathbb{R}^n)} \le \left\{\frac{\omega_{n-1}}{(2\pi)^n}\int_0^{\infty}\frac{r^{n-1}e^{-qtr^2}}{(1+r^2)^{(sq)/2}}dr\right\}^{1/q}\|f\|_{H^{s,q}(\mathbb{R}^n)}
\end{equation}
for all $t>0$, where $\omega_{n-1}>0$ denotes the surface area of the unit ball in $\mathbb{R}^n$. In particular, there holds
\begin{equation}
\|e^{t\Delta}f\|_{L^{\infty}(\mathbb{R}^n)} \le \left[\frac{2\log(e+1/t)}{(4\pi)^{n/2}\Gamma(n/2)}\left\{\frac{1}{\sqrt{e}}+\frac{1}{n}\left(\frac{e}{2q}\right)^{n/2}\right\}\right]^{1/q}\|f\|_{H^{n/q,q}(\mathbb{R}^n)},
\end{equation}
where $\Gamma$ is the gamma function given by $\Gamma(a) \coloneqq \int_0^{\infty}\lambda^{a-1}e^{-\lambda}d\lambda$ for $a>0$.
\end{theo}

\begin{proof}
Recall the $n$-dimensional version of \eqref{PI} and $\|\mathcal{F}^{-1}g\|_{L^{\infty}(\mathbb{R}^n)} \le (2\pi)^{-n/2}\|g\|_{L^1(\mathbb{R}^n)}$ to obtain
$$\|\mathcal{F}^{-1}g\|_{L^{q'}(\mathbb{R}^n)} \le \frac{1}{(2\pi)^{n(1/q-1/2)}}\|g\|_{L^q(\mathbb{R}^n)}$$
for all $g \in L^q(\mathbb{R}^n)$ and $1 \le q \le 2$ with the aid of the Riesz-Thorin interpolation theorem \cite{BL76}*{Theorem 1.1.1}, where $1/q+1/q'=1$. Since
$$e^{t\Delta}f=\frac{1}{(2\pi)^{n/2}}\mathcal{F}^{-1}[(1+|\cdot|^2)^{-s/2}e^{-t|\cdot|^2}]*\mathcal{F}^{-1}[(1+|\cdot|^2)^{s/2}\mathcal{F}f]$$
for all $s \in \mathbb{R}$, we see by the Young convolution inequality that
\begin{equation}
\begin{aligned}
\|e^{t\Delta}f\|_{L^{\infty}(\mathbb{R}^n)} &\le \frac{1}{(2\pi)^{n/2}}\|\mathcal{F}^{-1}[(1+|\cdot|^2)^{-s/2}e^{-t|\cdot|^2}]\|_{L^{q'}(\mathbb{R}^n)}\|\mathcal{F}^{-1}[(1+|\cdot|^2)^{s/2}\mathcal{F}f]\|_{L^q(\mathbb{R}^n)} \\
&\le \frac{1}{(2\pi)^{n/q}}\|(1+|\cdot|^2)^{-s/2}e^{-t|\cdot|^2}\|_{L^q(\mathbb{R}^n)}\|f\|_{H^{s,q}(\mathbb{R}^n)}.
\end{aligned}
\end{equation}
Thus the relation
$$\|(1+|\cdot|^2)^{-s/2}e^{-t|\cdot|^2}\|_{L^q(\mathbb{R}^n)}^q=\omega_{n-1}\int_0^{\infty}\frac{e^{-qtr^2}}{(1+r^2)^{(sq)/2}} \cdot r^{n-1}dr$$
yields the first inequality.

The second inequality is obtained by letting $s=n/q$. Indeed, noting that the substitution $\rho=qtr^2$ gives $dr/d\rho=1/(2qtr)$, we have
\begin{equation}
\int_{\sqrt{a}}^{\infty}\frac{r^{n-1}e^{-qtr^2}}{(1+r^2)^{n/2}}dr \le \int_{\sqrt{a}}^{\infty} \frac{e^{-qtr^2}}{r}dr =\int_{aqt}^{\infty}\frac{e^{-\rho}}{r}\frac{dr}{d\rho}d\rho =\frac{1}{2}\int_{aqt}^{\infty}\frac{e^{-\rho}}{\rho}d\rho =\frac{1}{2}E_1(aqt)
\end{equation}
for all $t,a \in (0,\infty)$. In addition, there holds
\begin{equation}
\int_0^{\sqrt{a}}\frac{r^{n-1}e^{-qtr^2}}{(1+r^2)^{n/2}}dr \le \int_0^{\sqrt{a}}r^{n-1}dr=\frac{a^{n/2}}{n}.
\end{equation}
As a similar calculation to the estimate \eqref{E12} leads to $E_1(aqt) \le (aq)^{-1}e^{aq/e}\log(e+1/t)$ for $a \le e/q$, we set $a=e/(2q)$ and use the fact $\omega_{n-1}=2\pi^{n/2}(\Gamma(n/2))^{-1}$ to obtain
\begin{equation}
\begin{aligned}
\frac{\omega_{n-1}}{(2\pi)^n}\int_0^{\infty}\frac{r^{n-1}e^{-qtr^2}}{(1+r^2)^{n/2}}dr  &\le \frac{2\pi^{n/2}}{\Gamma(n/2)} \cdot \frac{1}{(2\pi)^n}\left\{\frac{1}{2} \cdot \frac{2}{e} \cdot \sqrt{e}+\frac{1}{n}\left(\frac{e}{2q}\right)^{n/2}\right\}\log(e+1/t) \\
&=\frac{2\log(e+1/t)}{(4\pi)^{n/2}\Gamma(n/2)}\left\{\frac{1}{\sqrt{e}}+\frac{1}{n}\left(\frac{e}{2q}\right)^{n/2}\right\},
\end{aligned}
\end{equation}
which provides the second inequality.
\end{proof}

We remark that Theorem \ref{CSEq} certainly generalizes the estimate in Theorem \ref{CSE} with the bound $C \le (2\pi)^{-1/2}$. Indeed, we note that $\|f\|_{H^1(\mathbb{R}^2)}=\|f\|_{H^{1,2}(\mathbb{R}^2)}$ due to \eqref{PI}. Moreover, since $e<3$ and since $13/8<\sqrt{e}$ due to $1+\tau+\tau^2/2 \le e^{\tau}$, letting $n=q=2$ provides that
\begin{equation}
\left\{\frac{2}{4\pi \Gamma(1)}\left(\frac{1}{\sqrt{e}}+\frac{1}{2} \cdot \frac{e}{4}\right)\right\}^{1/2}=(2\pi)^{-1/2}\left(\frac{1}{\sqrt{e}}+\frac{e}{8}\right)^{1/2}<(2\pi)^{-1/2}\left(\frac{8}{13}+\frac{3}{8}\right)^{1/2}<(2\pi)^{-1/2}.
\end{equation}

\section{Optimality of the time-singularity: Proof of Theorem \ref{optCSE}}

Finally, we shall show Theorem \ref{optCSE}, which states that our estimate in Theorem \ref{CSE} is \textit{optimal} in the sense that the time-singularity $O((\log(1/t))^{1/2})$ as $t \to 0^+$ \textit{cannot} be improved.

\begin{proof}[Proof of Theorem \ref{optCSE}]
For $\lambda>1$, define the function $\eta_{\lambda} \in L^{\infty}(0,\infty)$ by setting
$$\eta_{\lambda} \coloneqq \left\{\begin{array}{cl}
1 & \text{on $(1,\lambda)$}, \\
0 & \text{on $(0,1] \cap [\lambda,\infty)$}.
\end{array}\right.$$
Then the function $f(\cdot,\lambda) \coloneqq \mathcal{F}^{-1}[|\cdot|^{-2}\eta_{\lambda}(|\cdot|)]$ satisfies
\begin{equation}
\begin{aligned}
\int_{\mathbb{R}^2}(1+|\xi|^2)|\mathcal{F}[f(\cdot,\lambda)](\xi)|^2d\xi &=\int_{\mathbb{R}^2}(1+|\xi|^2)|\xi|^{-4}\eta_{\lambda}(|\xi|)^2d\xi \\
&=2\pi \int_1^{\lambda}(1+r^2)r^{-4} \cdot rdr \\
&\le 4\pi \int_1^{\lambda}\frac{1}{r}dr=4\pi \log\lambda
\end{aligned}
\end{equation}
for all $\lambda>1$, which implies that $f(\cdot,\lambda) \in H^1(\mathbb{R}^2)$ with the estimate
$$\|f(\cdot,\lambda)\|_{H^1(\mathbb{R}^2)} \le 2\sqrt{\pi}(\log \lambda)^{1/2}.$$
Furthermore, as there holds
\begin{equation}
\begin{aligned}
e^{t\Delta}[f(\cdot,\lambda)](x) &=\frac{1}{2\pi}\int_{\mathbb{R}^2}e^{ix\cdot\xi}e^{-t|\xi|^2}\mathcal{F}[f(\cdot,\lambda)](\xi)d\xi \\
&=\frac{1}{2\pi}\int_{\mathbb{R}^2}e^{ix\cdot\xi}e^{-t|\xi|^2}|\xi|^{-2}\eta_{\lambda}(|\xi|)d\xi \\
&=\frac{1}{2\pi}\int_{1<|\xi|<\lambda}e^{ix\cdot\xi}e^{-t|\xi|^2}|\xi|^{-2}d\xi
\end{aligned}
\end{equation}
for all $\lambda>1$, $t>0$, and $x \in \mathbb{R}^2$, we observe that $e^{t\Delta}[f(\cdot,\lambda)](0) \in \mathbb{R}$ with the estimate
\begin{equation}
e^{t\Delta}[f(\cdot,\lambda)](0)=\int_1^{\lambda}e^{-tr^2}r^{-2} \cdot rdr \ge e^{-t\lambda^2}\int_1^{\lambda}\frac{1}{r}dr=e^{-t\lambda^2}\log \lambda.
\end{equation}
Hence, combining the above estimates with the present assumption yields
\begin{equation}
\begin{alignedat}{3}
e^{-t\lambda^2}\log \lambda &\le e^{t\Delta}[f(\cdot,\lambda)](0) &&\le \|e^{t\Delta}[f(\cdot,\lambda)]\|_{L^{\infty}(\mathbb{R}^2)} \\
&\le M(t)\|f(\cdot,\lambda)\|_{H^1(\mathbb{R}^2)} &&\le M(t) \cdot 2\sqrt{\pi}(\log \lambda)^{1/2}
\end{alignedat}
\end{equation}
for all $\lambda>1$ and $t>0$, which leads to
$$M(t) \ge \frac{1}{2\sqrt{\pi}}e^{-t\lambda^2}(\log \lambda)^{1/2}.$$
We choose $\lambda=(e+1/t)^{1/2}$ for $0<t<1/e$. Then it holds by $t\lambda^2=et+1<2$ that
$$M(t) \ge \frac{1}{2\sqrt{\pi}}e^{-2}(\log (e+1/t)^{1/2})^{1/2}=\frac{1}{2\sqrt{2\pi}e^2}(\log (e+1/t))^{1/2},$$
which yields the desired result. This completes the proof of Theorem \ref{optCSE}.
\end{proof}


\section*{Acknowledgment}

TO is grateful to Luis Vega and Nicola Visciglia for enlightening discussions.

YCH is partially supported by JSPS Invitational Fellowship for Research in Japan \#S24040. TO is partially supported by JSPS Grant-in-Aid for Scientific Research (S) \#JP24H00024. CS is partially supported by the ANR project Smooth ANR-22-CE40-0017. TT is partially supported by JSPS Grant-in-Aid for JSPS Fellows \#JP24KJ0122 and for Early-Career Scientists \#JP24K16954.

\bibliography{Critical-Sobolev-bib}

\end{document}